 \documentclass[final,1p,times]{elsarticle}
\usepackage{amsmath,amssymb,url}
\usepackage{enumitem} 
\numberwithin{equation}{section}

\usepackage{amssymb}
\usepackage{amsthm}
\usepackage{hyperref}
\usepackage{amsmath,amssymb,amsopn,amsfonts,mathrsfs,amsbsy,amscd}
\usepackage{longtable}
\journal{Communications of the Korean Mathematical Society}

\newcommand{\A}{{\cal A}}

\newcommand{\e}{\epsilon}

\newtheorem{definition}{Definition}[section]
\newtheorem{theorem}{Theorem}[section]
\newtheorem{proposition}{Proposition}[section]

\newtheorem{remark}{Remark}[section]
\numberwithin{equation}{section}

\begin{document}
\bibliographystyle{amsplain}
\begin{frontmatter}
\title{A characterization of generalized derivations and biderivations on some simple algebras}

\author[label2]{Hassan Oubba}
\address[label2]{Universit\'e Hassan II de Casablanca\\
		Facult\'e des sciences A\"in Chock -  D\'epartement de Math\'ematiques et Informatique\\
		B.P. 5366 Maarif, Casablanca, Maroc\\
        Laboratoire d'analyse Mathématique, Algèbre et Applications\\e-mail:
		 HASSAN.OUBBA@univh2c.ma}

\begin{abstract}

In this paper, we characterize the generalized derivations and biderivations of the simple seven-dimensional anticommutative algebra \(D\) and the seven-dimensional central simple commutative algebra \(C\). Furthermore, we prove that any local generalized derivation of these algebras is inherently a generalized derivation..
\end{abstract}
\end{frontmatter}
\section*{Introduction} \label{Introduction}
The notion of generalized derivation is due to Brešar \cite{br0}. Let $A$ be an algebra. A linear mapping $F : A \to A$ is called a \textbf{generalized derivation} if there exists a derivation $D : A \to A$ such that
$F(xy) = F(x)y + xD(y) \quad \text{for all } x, y \in A$.
A linear map $T : A \to A$ is called a \textbf{left centralizer} in case $T(xy) = T(x)y \quad \text{for all } x, y \in A$.
If $F$ is a generalized derivation determined by a derivation $D$, then $F - D$ is a left centralizer. Hence a generalized derivation is a sum of a derivation and a left centralizer. We say $F$ is \textbf{inner} if there exist $a, b \in A$ such that $F(x) = ax + xb \quad \text{for all } x \in A$. Obviously, if the algebra is unital then $F(x) = F(1)x + D(x)$ for all $x \in A$; in this case $F$ is inner if and only if $D$ is inner \cite{br0}. For results concerning generalized derivations we refer the reader to \cite{br0, al, hvala, lee, lee1}.

Over the past three decades, a significant line of inquiry has been to determine the relationship between the existence of such maps satisfying certain algebraic identities and the commutativity of the algebra $A$. This research program was initiated by the celebrated result of Posner \cite{pos}, who proved that the existence of a nonzero centralizing derivation on a prime ring forces the ring to be commutative. Posner's seminal work has inspired a vast body of literature investigating how derivations and generalized derivations influence the commutativity of prime and semiprime rings. Numerous authors have extended this line of investigation, a sample of which can be found in \cite{al, as, lee, lee1, hvala}, and the references therein. Kadison in \cite{kad} introduced and investigated a notion of local derivations. He proved that each
continuous local derivation from a von Neumann algebra into its dual Banach bimodule is a deriva-
tion. Semrl introduced a similar notion of 2-local derivations. He proved that any 2-local derivation
of the algebra $B(H)$ of all bounded linear operators on the infinite-dimensional separable Hilbert
space $H$ is a derivation \cite{sem}. After, numerous new results related to the description of local and
2-local derivations (see \cite{ayo2, ayo3, Lar, oub, oub2, oub4}). \\
The study of  biderivations traces its origins back to associative ring theory \cite{Bres1, BresMar2}. In particular, it was established that all biderivations on noncommutative prime rings are inner. Biderivations serve as valuable tools in analyzing commuting maps, with various studies determining commuting maps first by obtaining all biderivations of the algebras under investigation, as observed in \cite{bn, houb, oub, oub1, oub2, oub3, oub4,  Tang, Wang1, Wang2}.\\

In this work, we  study generalized derivations and biderivations of the simple finite-dimensional algebras
without finite basis of identities, constructed by Kislitsin in \cite{kis1, kis2}.\\ 
This paper is organized as follows.
In Section 2, we characterize all generalized derivations, local (and 2-local) generalized derivations, and biderivations of the simple algebra \(D\). The final section is dedicated to characterizing these same maps on the algebra \(C\). 

\section{Preliminaries}
\begin{definition}
   Let $A$ be an algebra. The set 
$$\Gamma(A)=\left\lbrace  \gamma \in End(A) \hspace{0.2 cm} : \hspace{0.2 cm} \gamma(x.y)=\gamma(x).y=x.\gamma(y), \hspace{0.2 cm} \forall x,y \in A \right\rbrace $$
is called the cenroid of $A$.
\end{definition}
\begin{definition}
\label{def1}
Let $(A,.)$ be an algebra. A linear map $D : A \rightarrow A$ is called a derivation, if 
$$D(x.y)=D(x).y+x.D(y)$$
for all $x,y \in A$.
\end{definition}
 We denote by $Der(A)$ the set of all derivations on $A$.
\begin{definition}
\label{def2}
Let $(A,.)$ be an algebra. A linear map $F : A \rightarrow A$ is called a generalized derivation, if there exist a derivation $D$ on $A$ such that
$$F(x.y)=F(x).y+x.D(y)$$
for all $x,y \in A$.
\end{definition}
Clearly, any derivation is a generalized derivation but the converse is, in general, not true. Hence, it is
 clear by the above definition that a generalized derivation composes a derivation and an element of the centroid.
 \begin{definition}
      Let $(\A,.)$ be a nonassociative algebra and let $\delta: \A \times \A \rightarrow \A$ be a bilinear map. 
 
 $\delta$ is called a biderivation of $(\A,.)$ if for any $x, y \in \A,$  $\delta(x,\,)$ and $\delta(\,,y)$ are derivations of $(\A,.), $ that means
 
 $$\delta(x,y.y')= \delta(x,y).y' + y.\delta(x,y')$$
 
 and 
 
 $$\delta(x.x',y)= \delta(x,y).x' + x.\delta(x',y),$$
 
 $\forall\, x,x',y,y' \in \A.$
 \end{definition}
  Denote by $BDer(A)$ the set of all biderivations on $A$ which is clearly a vector space.\\
A $\delta \in BDer(A)$ is called symmetric if $\delta(x,y)=\delta(y,x)$ for all $x,y \in A$, and is called skew-symmetric if $\delta(x,y)=-\delta(y,x)$ for all $x,y \in A$. Denote by $BDer_+(A)$ and $BDer_-(A)$ the subspace of all symmetric biderivations and all skew-symmetric biderivations on $A$ respectively.\\
For any $\delta \in BDer(A)$, we define two bilinear maps by 
$$\delta^+(x,y)=\delta(x,y)+\delta(y,x), \hspace{0.3 cm}\delta^-(x,y)=\delta(x,y)-\delta(y,x).$$
It is easy to see that $\delta^+ \in BDer_+(A)$ and $\delta^- \in BDer_-(A)$. Since $\delta =\frac{1}{2}(\delta^+ + \delta^-)$ it follows that
$$BDer(A)=BDer_+(A) \oplus BDer_-(A).$$
To characterize $BDer(A)$, we only need to characterize $BDer_+(A)$ and $BDer_-(A)$.
\section{A simple finite-dimensional algebra without finite basis of identities}
Let $D=\langle e,v_{1},v_{2},e_{11},e_{12},e_{22},p\rangle_{\mathbb{F}}$ be an algebra over a field $\mathbb{F}$ of characteristic zero, whose non-zero multiplication table for the basis elements is given by
$$v_{i}e_{ij}=-e_{ij}v_{i}=v_{j}, \quad v_{2}p=-pv_{2}=e, \quad v_{i}e=-ev_{i}=v_{i},$$
$$e_{ij}e=-ee_{ij}=e_{ij}, \quad pe=-ep=p$$
It is known that \(D\) is a simple anticommutative algebra lacking a finite basis of identities \cite{kis2}. Let \(a\) be an arbitrary element of \(D\), expressed as
$$a=a_{1}e+a_{2}v_{1}+a_{3}v_{2}+a_{4}e_{11}+a_{5}e_{12}+a_{6}e_{22}+a_{7}p$$
for some elements $a_{1}$, $a_{2}$, $a_{3}$, $a_{4},$ $a_{5}$, $a_{6}$, and $a_{7}$ in $\mathbb{F}$. Throughout the paper, let
$$\overline{a}=(a_{1},a_{2},a_{3},a_{4},a_{5},a_{6},a_{7})^{T}.$$
Conversely, if $v=(a_{1},a_{2},a_{3},a_{4},a_{5},a_{6},a_{7})^{T}$ is a column vector with $a_{1}$, $a_{2}$, $a_{3}$, $a_{4}$, $a_{5}$, $a_{6}$, and $a_{7}$ in $\mathbb{F}$
then, throughout the paper, we will denote by the element
$$a_{1}e+a_{2}v_{1}+a_{3}v_{2}+a_{4}e_{11}+a_{5}e_{12}+a_{6}e_{22}+a_{7}p;$$
i.e.,
$$\widehat{v}=a_{1}e+a_{2}v_{1}+a_{3}v_{2}+a_{4}e_{11}+a_{5}e_{12}+a_{6}e_{22}+a_{7}p.$$
Let $A$ be an algebra. A linear map $D:A \rightarrow A$ is called a derivation if
for any two elements $x,y\in A$
$$D(xy)=D(x)y+xD(y)$$
Our principal tool for the description of generalized derivations of $D$ is the following
proposition.

\begin{proposition}\label{pr1} \cite{far}
     A linear map $D: D \rightarrow D$ is a derivation if and only if the matrix $B$ of $D$ in the
standard basis  has the following form:
\begin{equation}
\label{form1}
B=\begin{pmatrix}
0 & 0 & 0 & 0 & 0 & 0 & 0 \\
0 & \alpha & 0 & 0 & 0 & 0 & 0 \\
0 & 0 & \alpha+\beta & 0 & 0 & 0 & 0 \\
0 & 0 & 0 & 0 & 0 & 0 & 0 \\
0 & 0 & 0 & 0 & \beta & 0 & 0 \\
0 & 0 & 0 & 0 & 0 & 0 & 0 \\
0 & 0 & 0 & 0 & 0 & 0 & -(\alpha+\beta)
\end{pmatrix}.
\end{equation}
Here the action of $D$ corresponds to multiplying the matrix by a column on the right.
\end{proposition}
\subsection{ Generalized derivations on $D$}
In this subsection we characterize generalized derivation and local (2-local) generalized derivations of $D$.\\
We give here our first result.
\begin{theorem}\label{th2}
    A linear map $F: D \rightarrow D$ is a generalized derivation if and only if the matrix of $F$ in the
standard basis  has the following form:
$$
\begin{pmatrix}
\lambda & 0 & 0 & 0 & 0 & 0 & 0 \\
0 & \gamma & 0 & 0 & 0 & 0 & 0 \\
0 & 0 & \gamma+\mu-\lambda & 0 & 0 & 0 & 0 \\
0 & 0 & 0 & \lambda & 0 & 0 & 0 \\
0 & 0 & 0 & 0 & \mu & 0 & 0 \\
0 & 0 & 0 & 0 & 0 & \lambda & 0 \\
0 & 0 & 0 & 0 & 0 & 0 & 3\lambda-(\gamma+\mu)
\end{pmatrix}\in M_7(\mathbb{F}).
$$
Here the action of $D$ corresponds to multiplying the matrix by a column on the right.
\end{theorem}
\begin{proof}
 Let $F$ be a linear map with the matrix in the basis $\lbrace e,v_1,v_2,e_{11},e_{12},e_{22},p\rbrace$ has the following form 
 $$\begin{pmatrix}
\lambda & 0 & 0 & 0 & 0 & 0 & 0 \\
0 & \gamma & 0 & 0 & 0 & 0 & 0 \\
0 & 0 & \gamma+\mu-\lambda & 0 & 0 & 0 & 0 \\
0 & 0 & 0 & \lambda & 0 & 0 & 0 \\
0 & 0 & 0 & 0 & \mu & 0 & 0 \\
0 & 0 & 0 & 0 & 0 & \lambda & 0 \\
0 & 0 & 0 & 0 & 0 & 0 & 3\lambda-(\gamma+\mu)
\end{pmatrix}$$
Then, for any $x\in D$, $F(x)=\lambda x +D(x)$ where $D$ is a linear map with the matrix in the basis $\lbrace e,v_1,v_2,e_{11},e_{12},e_{22},p\rbrace$ has the following form  
$$\begin{pmatrix}
0 & 0 & 0 & 0 & 0 & 0 & 0 \\
0 & \gamma-\lambda & 0 & 0 & 0 & 0 & 0 \\
0 & 0 & \gamma+\mu-2\lambda & 0 & 0 & 0 & 0 \\
0 & 0 & 0 & 0 & 0 & 0 & 0 \\
0 & 0 & 0 & 0 & \mu-\lambda & 0 & 0 \\
0 & 0 & 0 & 0 & 0 & 0 & 0 \\
0 & 0 & 0 & 0 & 0 & 0 & 2\lambda-(\gamma+\mu)
\end{pmatrix}$$
By Proposition \ref{pr1}, $D$ is a derivation on $D$. In the other hand $F(xy)=\lambda(xy)+D(xy)=(\lambda x)y+D(x)y+xD(y)=F(x)y+xD(y)$ for all $x,y \in D$. Therefore, $F$ is generalized derivation on $D$ associated to the derivation $D$. \\
 The proof of the only part is carried out by checking the generalized derivation property on the algebra $D$.
Let $F$ be a generalized derivation on $D$ and $A=(a_{i,j})_{i,j=1}^{7}$ be it's matrix. By the definition of generalized derivation and Proposition\ref{pr1}, there exists a matrix $B$ of the form \ref{form1}  such that,
$$\widehat{A\overline{xy}}=\widehat{A\overline{x}}y+x\widehat{B\overline{y}}, \quad \forall x,y \in D.$$
From the equalities $\widehat{A\overline{ee}}=\widehat{A\overline{e}}e+e\widehat{B\overline{e}}$ we have 
$$b_{21}=b_{31}=b_{41}=b_{51}=b_{61}=b_{71}=0.$$
From the equalities $\widehat{A\overline{v_1e}}=\widehat{A\overline{v_1}}e+v_1\widehat{B\overline{e}}$ we deduce
$$b_{22}=b_{11}+\alpha, \quad b_{12}=b_{32}=b_{42}=b_{52}=b_{62}=b_{72}=0.$$
From the equalities $\widehat{A\overline{v_2e}}=\widehat{A\overline{v_2}}e+v_2\widehat{B\overline{e}}$ we get
$$b_{33}=b_{11}+\alpha+\beta, \quad b_{13}=b_{23}=b_{43}=b_{53}=b_{63}=b_{73}=0.$$
From the equalities $\widehat{A\overline{e_{11}e}}=\widehat{A\overline{e_{11}}}e+e_{11}\widehat{B\overline{e}}$ we deduce
$$b_{44}=b_{11}, \quad b_{14}=b_{24}=b_{34}=b_{54}=b_{64}=b_{74}=0.$$
From the equalities $\widehat{A\overline{e_{12}e}}=\widehat{A\overline{e_{12}}}e+e_{12}\widehat{B\overline{e}}$ we have
$$b_{55}=b_{11}+\beta, \quad b_{15}=b_{25}=b_{35}=b_{45}=b_{65}=b_{75}=0.$$
From the equalities $\widehat{A\overline{e_{22}e}}=\widehat{A\overline{e_{22}}}e+e_{22}\widehat{B\overline{e}}$ we deduce
$$b_{66}=b_{11}, \quad b_{16}=b_{26}=b_{36}=b_{46}=b_{56}=b_{76}=0.$$
From the equalities $\widehat{A\overline{pe}}=\widehat{A\overline{p}}e+p\widehat{B\overline{e}}$ we deduce
$$b_{77}=b_{11}-(\alpha+\beta), \quad b_{17}=b_{27}=b_{37}=b_{47}=b_{57}=b_{67}=0.$$
As a result, by setting $\lambda=b_{11}, \gamma=\alpha+\lambda, \mu=\beta+\lambda$ we get the matrix from Theorem \ref{th2}. The proof is complete.    
\end{proof}
\begin{remark}
    It is easy to see that any generalized derivation of $D$ is a sum of $\lambda id$ and a derivation. Then, we can use the above theorem for prove that $\Gamma(D)=\lbrace f :D\to D\, : \, f(xy)=f(x)y=xf(y), \quad x,y \in D\rbrace = \lbrace \lambda id \,: \, \lambda \in \mathbb{F}\rbrace$.
\end{remark}

Let $A$ be an algebra. A linear map $\Delta:A \rightarrow A$ is called a local generalized derivation if, for any element
$x\in A,$ there exists a generalized derivation $F_x:A \rightarrow A$ such that $\Delta(x)=F_x(x)$.
\begin{theorem}\label{th3}
    Each local generalized derivation on the simple algebra $D$ is a generalized derivation.
\end{theorem}
\begin{proof}
     Let $\Delta$ be a local generalized derivation on $D,$ and let $A=(a_{i,j})_{i,j=1}^{7}$ be the matrix of $\Delta$. Then
$$\Delta(e)=\lambda^ee, \quad \Delta(v_{1})=\gamma^{v_{1}}v_{1}, \quad \Delta(v_{2})=(\gamma^{v_2}+\mu^{v_{2}}-\lambda^{v_{2}})v_{2},\quad \Delta(e_{11})=\lambda^{e_{11}}e_{11},$$ 
$$ \Delta(e_{1,2})=\mu^{e_{12}}e_{1,2},\quad \Delta(e_{22})=\lambda^{e_{22}}e_{22}, \quad \Delta(p)=(3\lambda^p-(\gamma^{p}+\mu^{p}))p,
 $$
and the remaining components of the matrix $A$ are equal to zero. At the same time,
\begin{equation}
    \label{equ1} \Delta(e+e_{11}+e_{22})=\Delta(e)+\Delta(e_{11})+\Delta(e_{22})
\end{equation}
and
\begin{equation}
    \label{equ1} \Delta(e+e_{11}+e_{22})=\lambda^{e+e_{11}+e_{22}}e+\lambda^{e+e_{11}+e_{22}}e_{11}+\lambda^{e+e_{11}+e_{22}}e_{22}
\end{equation}
Comparing coefficients at the basis element in (\ref{equ1}), we have $\lambda^{e+e_{11}+e_{22}}=\lambda^e=\lambda^{e_{11}}=\lambda^{e_{22}}.$ 
In the other hand
\begin{equation}\label{equ2}
\Delta(e+v_{1}+v_{2}+e_{1,2}+p)=\Delta(e)+\Delta(v_{1})+\Delta(v_{2})+\Delta(e_{1,2})+\Delta(p)
\end{equation}
and
\begin{eqnarray*}
 \Delta(e+v_{1}+v_{2}+e_{1,2}+p)&=&\lambda^{e+v_{1}+v_{2}+e_{1,2}+p} e+\gamma^{e+v_{1}+v_{2}+e_{1,2}+p}v_1\\
 &&+(\gamma^{e+v_{1}+v_{2}+e_{1,2}+p}+\mu^{e+v_{1}+v_{2}+e_{1,2}+p}-\lambda^{e+v_{1}+v_{2}+e_{1,2}+p})v_2\\
 &&+\mu^{e+v_{1}+v_{2}+e_{1,2}+p}e_{12}\\
 &&+(3\lambda^{e+v_{1}+v_{2}+e_{1,2}+p}-(\gamma^{e+v_{1}+v_{2}+e_{1,2}+p}+\mu^{e+v_{1}+v_{2}+e_{1,2}+p})p
\end{eqnarray*}
Comparing coefficients at the basis element in (\ref{equ2}), we obtain
\begin{eqnarray*}
   && \lambda^{e+v_{1}+v_{2}+e_{1,2}+p}=\lambda^e, \quad \gamma^{e+v_{1}+v_{2}+e_{1,2}+p}=\gamma^{v_1}, \quad \mu^{e+v_{1}+v_{2}+e_{1,2}+p}=\mu^{e_{12}}, \\
   && \gamma^{e+v_{1}+v_{2}+e_{1,2}+p}+\mu^{e+v_{1}+v_{2}+e_{1,2}+p}-\lambda^{e+v_{1}+v_{2}+e_{1,2}+p}=\gamma^{v_2}+\mu^{v_2}-\lambda^{v_2},\\
   &&3\lambda^{e+v_{1}+v_{2}+e_{1,2}+p}-(\mu^{e+v_{1}+v_{2}+e_{1,2}+p}+\mu^{e+v_{1}+v_{2}+e_{1,2}+p})=3\lambda^p-(\gamma^p+\mu^p).
\end{eqnarray*}
This implies that
$$\gamma^{v_2}+\mu^{v_2}-\lambda^{v_2}=\gamma^{v_1}+\mu^{e_{12}}-\lambda^e, \quad 3\lambda^p-(\gamma^p+\mu^p)=3\lambda^e-(\gamma^{v_1}+\mu^{e_{12}}).$$
The linear operator $\Delta$ has the matrix
$$A=
\begin{pmatrix}
\lambda^e & 0 & 0 & 0 & 0 & 0 & 0 \\
0 & \gamma^{v_{1}} & 0 & 0 & 0 & 0 & 0 \\
0 & 0 & \gamma^{v_{1}}+\mu^{e_{12}}-\lambda^e & 0 & 0 & 0 & 0 \\
0 & 0 & 0 & \lambda^e & 0 & 0 & 0 \\
0 & 0 & 0 & 0 & \mu^{e_{12}} & 0 & 0 \\
0 & 0 & 0 & 0 & 0 & \lambda^e & 0 \\
0 & 0 & 0 & 0 & 0 & 0 & 3\lambda^e-(\gamma^{v_1}+\mu^{e_{11}})
\end{pmatrix}.
$$
Hence, by Theorem \ref{th2}, $\Delta$ is a derivation. This completes the proof.
\end{proof}

We give another characterization of generalized derivations on the algebra $D$ in the following theorem.
Let $A$ be an algebra.
A (not necessarily linear) map $\Delta:A \rightarrow A$ is called a 2-local generalized derivation
if, for all elements $x$, $y\in A$ there exists a generalized derivation $F_{x,y}:A\rightarrow A$ such that $\Delta(x)=F_{x,y}(x)$ and
$\Delta(y)=F_{x,y}(y)$.
\begin{theorem} \label{th4}
 Each 2-local generalized derivation on the simple algebra $D$ is a generalized derivation.   
\end{theorem}
\begin{proof}
Let $\Delta$ be a 2-local generalized derivation of $D$ and $a\in \lbrace e,v_1,v_2,e_{11},e_{12},e_{22},p\rbrace$. Then, by definition, for every element $x \in D$, there exists a generalized derivation $F_{x,a}$ of $D$  such that
\begin{center}
$\Delta(x)=F_{x,a}(x), \hspace{0.2 cm} \Delta(a)=F_{x,a}(a)$.
\end{center}
By Theorem \ref{th2}, the matrix $M^{x,e_i}$ of the generalized derivation $F_{x,a}$ has the following matrix form:
$$M^{x,a}=\begin{pmatrix}
\lambda^{x,a} & 0 & 0 & 0 & 0 & 0 & 0 \\
0 & \gamma^{x,a} & 0 & 0 & 0 & 0 & 0 \\
0 & 0 & \gamma^{x,a}+\mu^{x,a}-\lambda^{x,a} & 0 & 0 & 0 & 0 \\
0 & 0 & 0 & \lambda^{x,a} & 0 & 0 & 0 \\
0 & 0 & 0 & 0 & \mu^{x,a} & 0 & 0 \\
0 & 0 & 0 & 0 & 0 & \lambda^{x,a} & 0 \\
0 & 0 & 0 & 0 & 0 & 0 &3\lambda^{x,a} -(\gamma^{x,a}+\mu^{x,a})
\end{pmatrix}$$
Let $y \in D$. Then there exists a generalized derivation $F_{y,e_i}$ of $D$ such that 
\begin{center}
$\Delta(y)=F_{y,a}(y), \hspace{0.3 cm} \Delta(a)=F_{y,a}(a).$
\end{center}
By Theorem \ref{th2}, the matrix $M^{y,a}$ of the derivation $F_{y,a}$ has the following matrix form:
$$M^{y,a}=\begin{pmatrix}
\lambda^{y,a} & 0 & 0 & 0 & 0 & 0 & 0 \\
0 & \gamma^{y,a} & 0 & 0 & 0 & 0 & 0 \\
0 & 0 & \gamma^{y,a}+\mu^{y,a}-\lambda^{y,a} & 0 & 0 & 0 & 0 \\
0 & 0 & 0 & \lambda^{y,a} & 0 & 0 & 0 \\
0 & 0 & 0 & 0 & \mu^{y,a} & 0 & 0 \\
0 & 0 & 0 & 0 & 0 & \lambda^{y,a} & 0 \\
0 & 0 & 0 & 0 & 0 & 0 &3\lambda^{y,a} -(\gamma^{y,a}+\mu^{y,a})
\end{pmatrix}$$
Since $F_{x,a}(a)=\Delta(a)=F_{y,a}(a)$, we have
\begin{gather*}
		  	 \begin{cases}
		  	\lambda^{x,e}=\lambda^{y,e} ;&\\
            \lambda^{x,e_{11}}=\lambda^{y,e_{11}} ;&\\
            \lambda^{x,e_{22}}=\lambda^{y,e_{22}} ;&\\
		  	\,\gamma^{x,v_1}=\gamma^{y,v_1};& \\
            \,\gamma^{x,v_2}+\mu^{x,v_2}-\lambda^{x,v_2}=\gamma^{y,v_2}+\mu^{y,v_2}-\lambda^{y,v_2};& \\
            \,\gamma^{x,v_2}=\gamma^{y,v_2};& \\
		  	\,\mu^{x,e_{12}}=\mu^{y,e_{12}}   ;&\\ 
            \,3\lambda^{x,p}-(\gamma^{x,p}+\mu^{x,p})=3\lambda^{y,p}-(\gamma^{y,p}+\mu^{y,p})
		  	 \end{cases}
\end{gather*}
That is
$$M^{y,a}=M^{x,a},$$
Therefore, for any $x\in D$, 
$$\Delta(x)=F_{y,a}(x),$$
that is $F_{x,a}$ does not depend on $x$. Hence, $\Delta$ is a generalized derivation by Theorem \ref{th2}. This completes the proof.
\end{proof}
\subsection{Biderivations of $D$}
In this section we prove that any biderivation of $D$ is the zero mao.
\begin{theorem}\label{thb1}
    The algebra $D$ has non zero skew-symmetric biderivation.
\end{theorem}
\begin{proof}
     Let $\delta^-$ be a skew-symmetric biderivation of the  algebras  $D$ and $x,y \in D$, such that $\overline{x}=(x_1,x_2,x_3,x_4,x_5,x_6,x_7)$ and $\overline{y}=(y_1,y_2,y_3,y_4,y_5,y_6,y_7)$. Then, by the bilinearity of $\delta^-$, we obtain, 
$$\delta^-(x,y)=x_1\delta^-_e(y)+x_2\delta^-_{v_1}(y)+x_3\delta^-_{v_2}(y)+x_{4}\delta^-_{e_{11}}(y)+x_{5}\delta^-_{e_{12}}(y)+x_{6}\delta^-_{e_{22}}(y)+x_{7}\delta^-_p(y).$$ 
By Proposition \ref{pr1}, the matrix $D_{a}$ of $\delta^-_{a}$, for $a=e,v_1,v_2,e_{11},e_{12},e_{22},p$ in the basis $\lbrace e,v_1,v_2,e_{11},e_{12},e_{22},p \rbrace$ is of the form
$$D_{a}=\begin{pmatrix}
0 & 0 & 0 & 0 & 0 & 0 & 0 \\
0 & \alpha^a & 0 & 0 & 0 & 0 & 0 \\
0 & 0 & \alpha^a+\beta^a & 0 & 0 & 0 & 0 \\
0 & 0 & 0 & 0 & 0 & 0 & 0 \\
0 & 0 & 0 & 0 & \beta^a & 0 & 0 \\
0 & 0 & 0 & 0 & 0 & 0 & 0 \\
0 & 0 & 0 & 0 & 0 & 0 & -(\alpha^a+\beta^a)
\end{pmatrix}.$$
Since $\delta^-$ is skew-symmetric, then by $\delta^-(a,a)=0$, for $a=e,v_1,v_2,e_{11},e_{12},e_{22},p$ we have 
\begin{equation}\label{eq1}
    \alpha^{v_1}=\beta^{e_{12}}=0, \quad \alpha^{v_2}=-\beta^{v_2}, \quad \alpha^p=-\beta^p.
\end{equation}
The equalities $\delta^-(e,a)=-\delta^-(a,e)$ for $a=v_1,e_{12}$ we deduce that 
\begin{equation}\label{eq2}
    \alpha^e=\beta^e=0.
\end{equation}
The equalities $\delta^-(v_1,a)=-\delta^-(a,v_1)$ for $a=e_{12},p$ we get that 
\begin{equation}
    \label{eq3} \beta^{v_1}=0, \quad \alpha^{v_1}=-\beta^{v_1} \quad \alpha^{e_{12}}=\alpha^{p}=0.
\end{equation}

The equalities $\delta^-(v_2,a)=-\delta^-(a,v_2)$ for $a=e_{12},p$ we get that 
\begin{equation}
    \label{eq4} \beta^{v_2}=0, \quad \alpha^{v_2}=-\beta^{v_2}.
\end{equation}
The equalities $\delta^-(e_{11},a)=-\delta^-(a,e_{11})$ for $a=e_{12},p$ we deduce that 
\begin{equation} \label{eq5}
   \beta^{e_{11}}=0, \quad \alpha^{e_{11}}=-\beta^{e_{11}}.
\end{equation}
The equalities $\delta^-(e_{12},a)=-\delta^-(a,e_{12})$ for $a=v_1,e_{22},p$ we deduce that 
\begin{equation}\label{eq6}
 \alpha^{e_{12}}=\beta^{e_{22}}=\beta^p=0,\quad    \alpha^{e_{12}}=-\beta^{e_{12}}.
\end{equation}
The equalities $\delta^-(e_{22},a)=-\delta^-(a,e_{22})$, for $a=v_1,e_{12}$, we have that 
\begin{equation}\label{eq7}
 \alpha^{e_{22}}=\beta^{e_{22}}.
\end{equation}
By comparing equations (\ref{eq1})-(\ref{eq7}), deduce that $$D_e=D_{v_1}=D_{v_2}=D_{e_{11}}=D_{e_{12}}=D_{e_{22}}D_p=\begin{pmatrix}
0 & 0 & 0 & 0 & 0 & 0 & 0 \\
0 & 0 & 0 & 0 & 0 & 0 & 0 \\
0 & 0 & 0 & 0 & 0 & 0 & 0 \\
0 & 0 & 0 & 0 & 0 & 0 & 0 \\
0 & 0 & 0 & 0 & 0 & 0 & 0 \\
0 & 0 & 0 & 0 & 0 & 0 & 0 \\
0 & 0 & 0 & 0 & 0 & 0 & 0
\end{pmatrix}$$
Therefore, $\delta^-$ is the zero map.
\end{proof}
\begin{theorem}\label{thb2}
    The algebra $D$ has non zero symmetric biderivation.
\end{theorem}
\begin{proof}
     Let $\delta^+$ be a skew-symmetric biderivation of the  algebras  $D$ and $x,y \in D$, such that $\overline{x}=(x_1,x_2,x_3,x_4,x_5,x_6,x_7)$ and $\overline{y}=(y_1,y_2,y_3,y_4,y_5,y_6,y_7)$. Then, by the bilinearity of $\delta^+$, we obtain, 
$$\delta^+(x,y)=x_1\delta^+_e(y)+x_2\delta^+_{v_1}(y)+x_3\delta^+_{v_2}(y)+x_{4}\delta^+_{e_{11}}(y)+x_{5}\delta^+_{e_{12}}(y)+x_{6}\delta^+_{e_{22}}(y)+x_{7}\delta^+_p(y).$$ 
By Proposition \ref{pr1}, the matrix $D_{a}$ of $\delta^+_{a}$, for $a=e,v_1,v_2,e_{11},e_{12},e_{22},p$ in the basis $\lbrace e,v_1,v_2,e_{11},e_{12},e_{22},p \rbrace$ is of the form
$$D_{a}=\begin{pmatrix}
0 & 0 & 0 & 0 & 0 & 0 & 0 \\
0 & \alpha^a & 0 & 0 & 0 & 0 & 0 \\
0 & 0 & \alpha^a+\beta^a & 0 & 0 & 0 & 0 \\
0 & 0 & 0 & 0 & 0 & 0 & 0 \\
0 & 0 & 0 & 0 & \beta^a & 0 & 0 \\
0 & 0 & 0 & 0 & 0 & 0 & 0 \\
0 & 0 & 0 & 0 & 0 & 0 & -(\alpha^a+\beta^a)
\end{pmatrix}.$$
The equalities $\delta^+(e,a)=-\delta^+(a,e)$ for $a=v_1,e_{12}$ we deduce that 
\begin{equation}\label{eq21}
    \alpha^e=\beta^e=0.
\end{equation}
The equalities $\delta^+(v_1,a)=-\delta^+(a,v_1)$ for $a=e_{12},p$ we get that 
\begin{equation}
    \label{eq31} \beta^{v_1}=0, \quad \alpha^{v_1}=-\beta^{v_1} \quad \alpha^{e_{12}}=\alpha^{p}=0.
\end{equation}

The equalities $\delta^+(v_2,a)=-\delta^+(a,v_2)$ for $a=e_{12},p$ we get that 
\begin{equation}
    \label{eq41} \beta^{v_2}=0, \quad \alpha^{v_2}=-\beta^{v_2}.
\end{equation}
The equalities $\delta^+(e_{11},a)=-\delta^+(a,e_{11})$ for $a=e_{12},p$ we deduce that 
\begin{equation} \label{eq51}
   \beta^{e_{11}}=0, \quad \alpha^{e_{11}}=-\beta^{e_{11}}.
\end{equation}
The equalities $\delta^+(e_{12},a)=-\delta^+(a,e_{12})$ for $a=v_1,e_{22},p$ we deduce that 
\begin{equation}\label{eq61}
 \alpha^{e_{12}}=\beta^{e_{22}}=\beta^p=0,\quad    \alpha^{e_{12}}=-\beta^{e_{12}}.
\end{equation}
The equalities $\delta^+(e_{22},a)=-\delta^+(a,e_{22})$, for $a=v_1,e_{12}$, we have that 
\begin{equation}\label{eq71}
 \alpha^{e_{22}}=\beta^{e_{22}}.
\end{equation}
By comparing equations (\ref{eq1})-(\ref{eq7}), deduce that $$D_e=D_{v_1}=D_{v_2}=D_{e_{11}}=D_{e_{12}}=D_{e_{22}}D_p=\begin{pmatrix}
0 & 0 & 0 & 0 & 0 & 0 & 0 \\
0 & 0 & 0 & 0 & 0 & 0 & 0 \\
0 & 0 & 0 & 0 & 0 & 0 & 0 \\
0 & 0 & 0 & 0 & 0 & 0 & 0 \\
0 & 0 & 0 & 0 & 0 & 0 & 0 \\
0 & 0 & 0 & 0 & 0 & 0 & 0 \\
0 & 0 & 0 & 0 & 0 & 0 & 0
\end{pmatrix}$$
Therefore, $\delta^+$ is the zero map.
\end{proof}
Since, $BDer(A)=BDer_+(A) \oplus BDer_-(A).$ Then, we have the following theorem
\begin{theorem}
  The algebra $D$ has non zero symmetric biderivation.   
\end{theorem}

\section{The simple algebra $C$}
Let $C=\langle 1,v_{1},v_{2},e_{11},e_{12},e_{22},p\rangle_{\mathbb{F}}$ be an algebra over a field $\mathbb{F}$ of characteristic $0$, where $1$ is unity and nonzero
products of basis elements 
\begin{equation} \label{fr2}
\{1,v_{1},v_{2},e_{11},e_{12},e_{22},p\}
\end{equation}
other than $1$ are defined as follows
$$v_{i}e_{ij}=e_{ij}v_{i}=v_{j}, \quad v_{2}p=pv_{2}=1. $$
Then the algebra $C$ is a simple central commutative algebra with no finite basis of identities \cite{kis1}. Let $a$ be an
element in $C$. Then we can write
$$a=a_{1}1+a_{2}v_{1}+a_{3}v_{2}+a_{4}e_{11}+a_{5}e_{12}+a_{6}e_{22}+a_{7}p$$
for some elements $a_{1}$, $a_{2}$, $a_{3}$, $a_{4},$ $a_{5}$, $a_{6}$, and $a_{7}$ in $\mathbb{F}$. Throughout the paper, let
$$\overline{a}=(a_{1},a_{2},a_{3},a_{4},a_{5},a_{6},a_{7})^{T}.$$
Conversely, if $v=(a_{1},a_{2},a_{3},a_{4},a_{5},a_{6},a_{7})^{T}$ is a column vector with $a_{1}$, $a_{2}$, $a_{3}$, $a_{4}$, $a_{5}$, $a_{6}$, and $a_{7}$ in $\mathbb{F}$
then, throughout the paper, we will denote by the element
$$a_{1}1+a_{2}v_{1}+a_{3}v_{2}+a_{4}e_{11}+a_{5}e_{12}+a_{6}e_{22}+a_{7}p;$$
i.e.,
$$\widehat{v}=a_{1}e+a_{2}v_{1}+a_{3}v_{2}+a_{4}e_{11}+a_{5}e_{12}+a_{6}e_{22}+a_{7}p.$$
Our principal tool for the description of generalized derivations and local (2-local) generalized derivations of $C$ is the following
proposition.

\begin{proposition}\label{pr2}\cite{far}
     A linear map $D: C \rightarrow C$ is a derivation if and only if the matrix $B$ of $D$ in the
standard basis  has the following form:
\begin{equation}
\label{form2}
B=\begin{pmatrix}
0 & 0 & 0 & 0 & 0 & 0 & 0 \\
0 & \alpha & 0 & 0 & 0 & 0 & 0 \\
0 & 0 & \alpha+\beta & 0 & 0 & 0 & 0 \\
0 & 0 & 0 & 0 & 0 & 0 & 0 \\
0 & 0 & 0 & 0 & \beta & 0 & 0 \\
0 & 0 & 0 & 0 & 0 & 0 & 0 \\
0 & 0 & 0 & 0 & 0 & 0 & -(\alpha+\beta)
\end{pmatrix}.
\end{equation}
Here the action of $D$ corresponds to multiplying the matrix by a column on the right.
\end{proposition}
\begin{theorem}\label{th21}
    A linear map $F: C \rightarrow C$ is a generalized derivation if and only if the matrix of $F$ in the
standard basis (\ref{fr2}) has the following form:
$$
\begin{pmatrix}
\lambda & 0 & 0 & 0 & 0 & 0 & 0 \\
0 & \gamma & 0 & 0 & 0 & 0 & 0 \\
0 & 0 & \gamma+\mu-\lambda & 0 & 0 & 0 & 0 \\
0 & 0 & 0 & \lambda & 0 & 0 & 0 \\
0 & 0 & 0 & 0 & \mu & 0 & 0 \\
0 & 0 & 0 & 0 & 0 & \lambda & 0 \\
0 & 0 & 0 & 0 & 0 & 0 & 3\lambda-(\gamma+\mu)
\end{pmatrix}\in M_7(\mathbb{F}).
$$
Here the action of $D$ corresponds to multiplying the matrix by a column on the right.
\end{theorem}
\begin{proof}
    The proof of this theorem is similar to the proof of Theorem \ref{th2}
\end{proof}
Like in the case of algebra $D$, the same result holds for algebra $C$.
\begin{theorem}
  Each local (2-local) derivation on the simple algebra $C$ is a derivation.  
\end{theorem}
\begin{proof}
    The proof of this theorem is similar to the proofs of Theorems \ref{th3} and \ref{th4}.
\end{proof}
As in the case of algebra D, in algebra C as well, all biderivations are zero.
\begin{theorem}
    Any biderivations on $C$ is the zero map.
\end{theorem}
\begin{proof}
    The proof of this theorem is similar to the proofs of Theorems \ref{thb1} and \ref{thb2}.
\end{proof}

\vspace{1cm}
\noindent \textbf{Acknowledgements:} The authors thank the referees for their valuable comments that contributed to a sensible improvement of the paper.

\section*{Declarations}
\noindent {\bf Conflict of interest statement.} Not applicable
\section*{Data availability}
The paper has no associated data.


\begin{thebibliography}{99}

 \bibitem{al}
 E. Albas, N. Arga¸c, Generalized derivation on prime rings, Algebra Colloq. 11 (2004)
 399–410.
 \bibitem{as}
 M. Ashraf, A. Ali, S. Ali, Some commutativity theorems for rings with generalized
 derivations, Southeast Asian Bull. Math. 31 (2007) 415–421.
 \bibitem{ayo2}
Ayupov, Sh. Kudaybergenov, K. Local derivations on finite-dimensional Lie algebras. Linear Algebra and its Applications 493 (2016) 381$–$398.

\bibitem{ayo3}
Ayupov, Sh. Kudaybergenov, K. Omirov, B. Local and 2-local derivations and automorphisms on simple Leibniz algebras. Bulletin of the Malaysian Mathematical Sciences Society 43 (3) (2020) 2199$–$2234.
\bibitem{bn}
Benayadi, S.  Oubba, H. Nonassociative algebras of biderivation-type. Linear Algebra and its Applications, 2024, vol. 701, p. 22-60.

\bibitem{br0}
 Bresar, Matej. "On the distance of the composition of two derivations to the generalized derivations." Glasgow Mathematical Journal 33.1 (1991): 89-93.
 \bibitem{Bres1}
Bre\v{s}ar, M.:
 Centralizing mapping and derivations in prime rings. J. Algebra. 156, (1993) 385$-$394.


\bibitem{BresMar2}
Bre\v{s}ar, M., Martindale 3rd, W.S., Miers, C.R.:
 Centralizing maps in prime rings with involution. J. Algebra. 161, (1993) 342$-$357.
 \bibitem{far}
   Farhodjon, Arzikulov, Urinboyev Furkat, and Ergasheva Shahlo. "A characterization of derivations and automorphisms on some simple algebras." Ural Mathematical Journal 8.2 (15) (2022): 46-58.

   \bibitem{isa1}
   Isaev, I. M., and A. V. Kislitsin. "Example of simple finite dimensional algebra with no finite basis of its identities." Communications in Algebra 41.12 (2013): 4593-4601.
   \bibitem{isa2}
   Isaev, I. M., and A. V. Kislitsin. "Example of simple finite dimensional algebra with no finite basis of its identities." Communications in Algebra 41.12 (2013): 4593-4601.
\bibitem{kis1}
 Kislitsin, Alexey Vladimirovich. "An example of a central simple commutative finite-dimensional algebra with an infinite basis of identities." Algebra and Logic 54.3 (2015): 204-210.
\bibitem{kis2}
 Kislitsin, A. V. (2017). Simple finite-dimensional algebras without finite basis of identities. Siberian Mathematical Journal, 58(3), 461-466.



\bibitem{hvala}
Hvala, Generalized derivations in rings. Commun Algebra
 26(4):1147–1166
 
\bibitem{kad}
Kadison, R.V. Local derivations, Journal of Algebra, 130, (1990), 494-509.


\bibitem{lee}
Lee TK(1999) Generalized derivations of left faithful rings. Commun Algebra 27(8):4057–4073
\bibitem{lee1}
 Lee TK, Shiue WK (2001) Identities with generalized derivations. Commun Algebra 29(10):4437–4450



 \bibitem{Lar}
 Larson, D.R. Sourour, A.R Local derivations and local automorphisms of B(X), Proceedings of Symposia in Pure Mathematics, 51 (1990), 187-194.

\bibitem{houb}
Oubba, H. Biderivations of some evolution algebras. Asian-European Journal of Mathematics (2026): 2650050.

\bibitem{oub}
Oubba, H. Generalized quaternion algebras. Rendiconti del Circolo Matematico di Palermo Series 2, 2023, vol. 72, no 8, p. 4239-4250.
\bibitem{oub1}
Oubba, H. Biderivations, commuting linear maps, post-Lie algebra structure on solvable Lie algebras of maximal rank. Communications in Algebra, 2024, p. 1-10.
\bibitem{oub2}
Oubba, Hassan. "Local (2-Local) derivations and automorphisms and biderivations of complex $\omega$-Lie algebras." Le Matematiche 79.1 (2024): 135-150.
\bibitem{oub3}
Bajo, Ignacio, Saïd Benayadi, and Hassan Oubba. "Non-existence of Symmetric Biderivations on Finite-Dimensional Perfect Lie Algebras: I. Bajo et al." Mediterranean Journal of Mathematics 22.6 (2025): 162.
\bibitem{oub4}
Oubba H.  On dual quaternion algebra.  Commun. Korean Math. Soc. 2025;40:763-776.  https://doi.org/10.4134/CKMS.c240275







\bibitem{sem}
 Semrl P., Local automorphisms and derivations on  B(H), Proceedings of the American Mathematical Society, 125 (1997), 2677-2680.




\bibitem{pos}
Posner, Edward C. "Derivations in prime rings." Proceedings of the American Mathematical Society 8.6 (1957): 1093-1100.
\bibitem{Tang}
Tang, X.:
 Biderivations of finite dimensional complex simple Lie algebras. Linear and Multilinear Algebra. 66,  (2018) 250$-$259.
\bibitem{Wang1}
Wang, D., Yu, X.:
 Biderivations and Linear commuting maps on the Schr\"{o}dinger-Virasoro Lie algebras , Comm. Algebra. 41,  (2013) 2166$-$2173
\bibitem{Wang2}
Wang, D., Yu, X., Chen, Z.:
 Biderivations of the parabolic subalgebra of simple Lie algebras, Comm. Algebra. 39,  (2011) 4097$-$4104



\end{thebibliography}
\end{document}